\documentclass{article}
\usepackage{amsmath,amssymb,amsthm}
\newtheorem{Thm}{Theorem}
\newtheorem{Lem}{Lemma}

\theoremstyle{definition}

\newtheorem{Rem}{Remark}

\begin{document}

\title{A remark on the uniqueness of positive solutions to semilinear elliptic equations with double power nonlinearities}
\author{Shinji Kawano  \\ Department of Mathematics \\ Hokkaido University \\ Sapporo  060-0810, Japan}
\date{}
\maketitle

\begin{abstract}

We consider the uniqueness of positive solutions to
\begin{equation}
 \begin{cases}
   \triangle u- \omega u+u^p-u^{2p-1}=0  & \text{in  $\mathbb{R}^n$},\\
   \displaystyle \lim_{\lvert x \rvert \to \infty} u(x)  =0. \label{introduction}
 \end{cases}  
\end{equation}
It is known that for fixed $p>1$, a positive solution to \eqref{introduction}  exists if and only if $\omega \in (0, \omega_p)$, 
where $\omega_p:=\dfrac{p}{(p+1)^2}$.
We deduce the uniqueness in the case where $\omega$ is close to $\omega_p$,
from the argument in the classical paper by Peletier and Serrin~\cite{PS},
thereby recovering a part of the uniqueness result of Ouyang and Shi~\cite{OS} 
for all $\omega \in (0, \omega_p)$.
\end{abstract}

\section{Introduction}

We shall consider a boundary value problem
\begin{equation}
 \begin{cases}
   u_{rr}+ \dfrac{n-1}{r}u_r -\omega u+u^p-u^{2p-1}=0 & \text{for $r>0$}, \\
   u_r(0)=0, \\
   \displaystyle \lim_{r \to \infty} u(r) =0, \label{b}
 \end{cases} 
\end{equation} 
where $n \in \mathbb{N}$, $p>1$ and $\omega >0$.
The above problem arises in the study of 
\begin{equation}
 \begin{cases}
   \triangle u- \omega u+u^p-u^{2p-1}=0  & \text{in  $\mathbb{R}^n$},\\
   \displaystyle \lim_{\lvert x \rvert \to \infty} u(x)  =0. \label{a}
 \end{cases}  
\end{equation}
Indeed, the classical work of Gidas, Ni and Nirenberg~\cite{G1,G2} tells us that any positive solution to \eqref{a} 
is radially symmetric. 
On the other hand, for a solution $u(r)$ of \eqref{b}, $v(x):=u(\lvert x \rvert)$ is a solution to \eqref{a}.  

The condition to assure the existence of positive solutions to \eqref{a} (and so \eqref{b}) was given by 
Berestycki and Lions~\cite{B1} and Berestycki, Lions and Peletier~\cite{B2}:
A solution to \eqref{b} with fixed $p>1$ exists if and only if $\omega \in (0, \omega_p)$, where 
\begin{equation*}
\omega_p=\dfrac{p}{(p+1)^2}.
\end{equation*}
We shall review what this $\omega_p$ is for in Section 2.
Throughout this paper, \textit{a solution} means a classical solution.

Uniqueness of positive solutions to \eqref{b} had long remained unknown.
Finally in 1998 Ouyang and Shi~\cite{OS} proved uniqueness for \eqref{b} with all $\omega \in (0, \omega_p)$, $p>1$.
See also Kwong and Zhang~\cite{KZ}.

In this present paper, we prove that for $\omega$ close to $\omega_p$, the uniqueness result is obtained directly 
from the classial result given by Peletier and Serrin~\cite{PS} in 1983.
For another attempt to obtain the uniqueness when $\omega$ is close to $\omega_p$, see Mizumachi~\cite{M}.
Our result of the present paper is the following:
\begin{Thm}
Let $n \in \mathbb{N}$, $p>1$ and $\omega\in[a_p,\omega_p)$, where $a_p:=\dfrac{p(7p-5)}{4(p+1)(2p-1)^2}$.
Then \eqref{b} has exactly one positive solution.  \label{thm}
\end{Thm}

\begin{Rem}
Note that 
\begin{equation*}
0<a_p<\omega_p=\dfrac{p}{(p+1)^2}, \qquad p>1.
\end{equation*}
\end{Rem}

In the next section we clarify the definitions of $\omega_p$ and $a_p$ from the point of view from~\cite{PS}.

\section{Study of the nonlinearity as a function}

In this section, we study the properties of the function $f(u):=-\omega u+u^p-u^{2p-1}$ in $(0, \infty)$,
where $\omega>0$ and $p>1$ are given constants.

First we define $F(u):=\displaystyle \int_0^u f_{\omega,p}(s)ds$, and by a direct calculation we have
\begin{align}
   F(u) &= -\frac{\omega}{2}u^2+\frac{u^{p+1}}{p+1}-\frac{u^{2p}}{2p} \notag   \\
        &= \frac{u^2}{2p(p+1)}\left[ -\omega p(p+1)+2pu^{p-1}-(p+1)u^{2(p-1)}\right].  \label{c}
\end{align}
There are two cases of concern:
\begin{itemize}
\item[(a)]~~~~~ $\omega < \omega_p$ $\Longleftrightarrow$ $F$ has two zeros in $(0,\infty)$.
\item[(b)]~~~~~ $\omega \ge \omega_p$ $\Longleftrightarrow$ $F$ has at most one zero in $(0,\infty)$.
\end{itemize}

The condition to assure the existence of positive solutions of \eqref{b} given in~\cite{B1,B2} is the following;
\begin{Lem}
The problem \eqref{b} has a positive solution if and only if both of the following hypotheses are fulfilled:
\begin{itemize}
\item[(H1)] $\displaystyle \lim_{u \to +0}\frac{f(u)}{u}$ exists and is negative,
\item[(H2)] $F(\delta)>0$ for some positive constant $\delta$.
\end{itemize} \label{key}
\end{Lem}
\begin{Lem}
The problem \eqref{b} has a positive solution if and only if 
\begin{equation*}
\omega \in (0,\omega_p)   
\end{equation*}    
for $p>1$.
\label{condexlem}
\end{Lem}
\begin{proof}
(H1) is equivalent to the condition $\omega > 0$.
(H2) is equivalent to the condition (a) above.
\end{proof}

This is the origin of $\omega_p$. Next we turn to the exponent $a_p$.

As a preparation, we calculate the derivatives of $f(u)=-\omega u+u^p-u^{2p-1}$:
\begin{align*}
  f'(u)    &= -\omega+pu^{p-1}-(2p-1)u^{2(p-1)},  \\
  f''(u)   &= 2(p-1)(2p-1)u^{p-2}\left[ \dfrac{p}{2(2p-1)} -u^{p-1} \right].
\end{align*}

We shall introduce four positive constants $\alpha$, $b$, $c$ and $\beta$. 
\begin{itemize}
\item Let $\alpha$ denote the unique zero of $f''$ in $(0,\infty)$: 
$\displaystyle \alpha=\left[ \frac{p}{2(2p-1)}\right]^{\frac{1}{p-1}}$.
\item Let $b$ denote the first zero of $f$ in $(0,\infty)$:
$\displaystyle b=\left[ \frac{1-\sqrt{1-4\omega}}{2}\right]^{\frac{1}{p-1}}$.
\item Let $c$ denote the last zero of $f$ in $(0,\infty)$:
$\displaystyle c=\left[ \frac{1+\sqrt{1-4\omega}}{2}\right]^{\frac{1}{p-1}}$.
\item Let $\beta$ denote the first zero of $F$ in $(0,\infty)$:
$\displaystyle \beta=\left[ \frac{p}{p+1} \left(1-\sqrt{1-\frac{(p+1)^2}{p}\omega} \right)\right]^{\frac{1}{p-1}}$.
\end{itemize}
It is easy to check that 
\begin{equation}
  \beta \in (b, c) \label{ineq}  
\end{equation}
either by observing the graphs or by a straightforward calculation.
From \eqref{ineq} we deduce
\begin{equation}
f(\beta)>0,\label{z}
\end{equation} 
which will be used later.

We are not able to give a clear explanation on the relation between $\alpha$ and $\beta$.
\begin{Lem}
The condition $\alpha \le \beta$ is equivalent to 
$\omega \ge a_p=\displaystyle \frac{p(7p-5)}{4(p+1)(2p-1)^2}$. \label{eta}
\end{Lem}
\begin{proof}
A simple calculation.
\end{proof}
This is where our $a_p$ comes into play. In the next section, we see what this condition stands for.

\section{Proof of Theorem~\ref{thm}.}

First we state the result by Peletier and Serrin~\cite{PS}, which assures the uniqueness of solutions of \eqref{b}.
\begin{Lem}
Let $f$ satisfy (H1-3),
where (H1), (H2) are in Lemma~\ref{key}., and (H3) is the following:  
\begin{itemize}
\item[(H3)] $G(u):=\displaystyle \frac{f(u)}{u-\beta}$ is nonincreasing in 
$(\beta, c)$.
\end{itemize}
Then \eqref{b} has exactly one positive solution. \label{lem}
\end{Lem}
Now we are in a position to prove Theorem~\ref{thm}.

\begin{proof}[Proof of Theorem~\ref{thm}]
We shall see that for $\omega\in[a_p,\omega_p)$, (H1-3) are satisfied.
It is enough to show that if $\omega \ge a_p$, then 
\begin{equation}
k(u):=f'(u)(u-\beta )-f(u)\le 0 \qquad \text{in}  ~~(\beta, c).\label{reallast}
\end{equation}
To prove \eqref{reallast} we calculate the derivative of $k(u)$
\begin{equation*}
k'(u)=f''(u)(u-\beta),
\end{equation*}
and note that
\begin{alignat*}{2}
  f''(u) &> 0 & \qquad & \text{in} ~(0,\alpha); \\
  f''(u) &< 0 & \qquad & \text{in} ~(\alpha,\infty).
\end{alignat*}
So if $\alpha\le \beta$ (i.e. $\omega \ge a_p$, see Lemma~\ref{eta}), 
then $k'(u)<0$ in $(\beta, c)$, i.e. $k$ is decreasing in the interval. 
Therefore
\begin{equation*}  
 k(u)<k(\beta)=-f(\beta)<0 \qquad \text{in~~  $(\beta, c)$},
\end{equation*}
where the last inequality follows by \eqref{z}.

This proves \eqref{reallast} and completes the proof.
\end{proof}
If $\alpha>\beta$, we need to check that $k(\alpha)\le 0$, i.e.
\begin{equation}
\alpha -\frac{f(\alpha)}{f'(\alpha)}\le \beta.   \label{last}
\end{equation}
This condition provides an implicit relation between $\omega$ and $p$.
Besides, 
\begin{Rem}
The condition~\eqref{last} does not cover all $\omega \in (0,\omega_p)$.
That is for $\omega$ close to zero, $\alpha -\dfrac{f(\alpha)}{f'(\alpha)}> \beta$.  
\end{Rem} 
\begin{proof}
The left hand side of \eqref{last} is estimated from below as
\begin{align*}
 \alpha-\frac{f(\alpha)}{f'(\alpha)} &= \frac{(p-1){\alpha}^p(1-2{\alpha}^{p-1})}{-\omega +p{\alpha}^{p-1}-(2p-1){\alpha}^{2(p-1)}} \\
                                     &> \frac{(p-1){\alpha}^p(1-2{\alpha}^{p-1})}{p{\alpha}^{p-1}-(2p-1){\alpha}^{2(p-1)}}>0,
\end{align*}
for all $\omega\in(0,\omega_p)$, 
whereas the right hand side $\beta$ decreases to zero as $\omega$ decreases to zero.
\end{proof}
When $\omega$ is close to zero, a very delicate observation is needed. 
See Ouyang and Shi~\cite{OS} for details.

\end{document}